\documentclass{aims}
\usepackage{amsmath}
  \usepackage{paralist}
  \usepackage{graphics} 
  \usepackage{epsfig} 
\usepackage{graphicx}  \usepackage{epstopdf}
 \usepackage[colorlinks=true]{hyperref}
\usepackage[english]{babel}
\usepackage{xcolor,colortbl}
\usepackage{braket}

\def\F{\mathbf F}

\def\cA{\mathcal A}
\def\cB{\mathcal B}
\def\cC{\mathcal C}

\def\cK{\mathcal K}
\def\cL{\mathcal L}

\def\Im{\mbox{\rm Im}}

\def\Alt{\mbox{\rm Alt}}
\def\Sym{\mbox{\rm Sym}}
\def\dim{\mbox{\rm dim}}

\def\FF{{\mathbb F}}

\def\f2{{\mathbb F}_{2}}

\newcommand{\GL}{\mbox{\rm GL}}

\newcommand{\gd}{\delta}

\newcommand{\gl}{\lambda}

\newcommand{\gr}{\rho}
\newcommand{\gs}{\sigma}

\newcommand{\gt}{\tau}

\def\Im{\mathrm{Im}}

\def\F2{\mathbb{F}_{\hspace{-0.7mm}2}}

\hypersetup{urlcolor=blue, citecolor=red}

\def\currentvolume{X}
 \def\currentissue{X}
  \def\currentyear{200X}
   \def\currentmonth{XX}
    \def\ppages{X--XX}
     \def\DOI{10.3934/xx.xx.xx.xx}

\newtheorem{theorem}{Theorem}[section]
\newtheorem{corollary}{Corollary}

\newtheorem{lemma}[theorem]{Lemma}
\newtheorem{proposition}{Proposition}

\theoremstyle{definition}
\newtheorem{definition}[theorem]{Definition}
\newtheorem{remark}{Remark}

\title[A note on some algebraic trapdoors for block ciphers] 
      {A note on some algebraic trapdoors for block ciphers}

\author[Marco Calderini]{}


 \email{marco.calderini@uib.no}

\begin{document}
\maketitle

\centerline{\scshape Marco Calderini}
\medskip
{\footnotesize
 \centerline{Department of Informatics, University of Bergen, Norway}
} 

%

\bigskip

 \centerline{
 }

\begin{abstract}
We provide sufficient conditions to guarantee that a translation based cipher is not vulnerable with respect to the partition-based trapdoor. This trapdoor has been introduced, recently, by Bannier et al. (2016) and it generalizes that introduced by Paterson in 1999. Moreover, we discuss the fact that studying the group generated by the round functions of a block cipher may not be sufficient to guarantee security against these trapdoors for the cipher.
\end{abstract}

\section{Introduction}

In the last years, since the work \cite{coppersmith1975generators} of Coppersmith and Grossman, much attention has been devoted to the group generated by the round functions of a block cipher. In this context, Paterson \cite{CGC-cry-art-paterson1} showed that the imprimitivity of the group can be exploited to construct a trapdoor. By a trapdoor we mean a hidden algebraic structure in the cipher design that would allow an attacker (with the knowledge of the trapdoor) to break it easily. In \cite{CGC-cry-art-carantisalaImp} Caranti, Dalla Volta and Sala introduced the definition of translation based cipher, which contains well-known ciphers like AES \cite{daemen2002design}, SERPENT \cite{CGC-cry-art-serpent} and PRESENT \cite{CGC-cry-art-PRESENT}. For this class of ciphers, in \cite{CGC-cry-art-carantisalaImp} and \cite{calderinisalerno}, the authors provided cryptographic conditions on the S-Boxes and the mixing layer, in order to guarantee the primitivity of the group generated by the round functions of the cipher.

In a recent work \cite{bannier2016partition}, inspired by the partition cryptanalysis developed in \cite{harpes1997partitioning}, the authors introduce the partition-based trapdoor. This type of trapdoor generalizes that introduced by Paterson. Moreover, the authors give an example of a (toy) block cipher which is not vulnerable with respect to linear \cite{linear} and differential attacks \cite{CGC2-cry-art-biham1991differential}, but that can be broken, easily, using the structure of the trapdoor. {A more sophisticated way to use such a weakness is given in \cite{bannier2017partition}}.

{The principal aim of this work is to investigate an open question left by Paterson in his work \cite{CGC-cry-art-paterson1}, that is, if it might be possible to have the case where the round functions generate a primitive group but the subgroup generated by the $\ell$-round cipher itself has a block structure. In particular, we want to find cryptographic properties that could avoid such a threat.}

In this work, we give some conditions on the S-boxes and the mixing layer of a translation based cipher, in order to avoid the partition-based trapdoor. From this result, we are able to give a security proof for the group of encryption functions of a cipher with independent round-keys. 


The paper is organized as follows. In Section 2, we recall some definitions and a series of properties and already known results. In Section 3, we show how we can avoid the partition-based trapdoor on a translation based cipher.
Finally, in Section 4, we discuss the fact that studying the group generated by the round functions of a block cipher may not be sufficient to guarantee security against these trapdoors for the cipher, {and we show that avoiding the partition-based trapdoor we can give some properties of the group generated by an $\ell$-round cipher with independent round-keys}. We report our conclusions and some final remarks in Section 5.
\section{Preliminaries and notation}

Let $\cC$ be a block cipher acting on a message space $V = (\FF_2)^d$, for some $d\ge 1$ (we suppose that $V$ coincides with the ciphertext space). Let $\cK$ be its key space. Then any key $k \in\cK$ individuates a permutation $\gt_k$ on the space $V$ and our cipher is given by the set
$$\cC=\{\gt_k\mid k \in\cK\}.$$

We are interested in determining the properties of the group
$\Gamma(\cC)=\langle \gt_k\mid k \in\cK\rangle$. Unfortunately, the study of $\Gamma(\cC)$ is a difficult task, in general. Most modern block ciphers are iterated ciphers, i.e., obtained by a composition of several key-dependent permutations, called rounds. This allows to investigate an other permutation group related to $\cC$. For an iterated block cipher $\cC$ each $\gt_k$ is a composition
of some permutations of $V$, say $\gt_{k,1}, ... , \gt_{k,\ell}$. For any round $h$, let 
$$\Gamma_h(\cC)=\langle \gt_{k,h}\mid k \in\cK \rangle,$$
therefore, we can define the group containing $\Gamma(\cC)$ generated by the round functions
$$\Gamma_\infty(\cC)=\langle \Gamma_h(\cC)\mid h=1,...,\ell \rangle.$$

\subsection{Translation based ciphers}
Here we consider translation based ciphers, introduced in \cite{CGC-cry-art-carantisalaImp}. This class of iterated block ciphers includes some well-known ciphers,
as for instance AES \cite{daemen2002design} and SERPENT  \cite{CGC-cry-art-serpent}.

We first fix the notation, in order to recall the definition of a translation based cipher $\cC$. Let $m, b > 1$ and
$$
V=V_1\oplus\dots\oplus V_b,
$$
where each $V_i$ is isomorphic to $(\FF_2)^m$. We will denote by $\Sym(V)$ the symmetric group on $V$. Given $v \in V$, we write $\gs_v\in \Sym(V )$ for the translation of $V$ mapping $x$ to $x + v$. We denote by $T(V)$ the group of all translations of $V$. We will write the action of $g \in \Sym(V)$ on an element $v\in V$ as $vg$.

For any $v\in V$, we will write $v=v_1\oplus\dots \oplus v_b$, where $v_i\in V_i$. Also, we consider the projections $\pi_i:V\to V_i$ mapping $v\mapsto v_i$.

Any $\gamma\in \Sym(V)$ that acts as $v \gamma=v_1\gamma_1\oplus\dots\oplus v_b\gamma_b$, for some $\gamma_i\in \Sym(V_i)$, is called {\em bricklayer transformation} (a ``parallel map'') and any $\gamma_i$'s is a {\em brick}. Traditionally, the maps $\gamma_i$'s are called S-boxes and $\gamma$ a ``parallel S-box''. A linear map $\gl:V\to V$ is traditionally said a ``Mixing Layer'' when used in composition with parallel maps. 
For any $I\subset \{1,...,b\}$, with $I\ne \emptyset$ and $I\ne \{1,...,b\}$, we define $\bigoplus_{i\in I} V_i$ a {\em wall}. 
\begin{definition}
A linear map $\gl\in \GL(V)$ is called a {\em proper mixing layer} if no wall is invariant under $\gl$.
\end{definition}

We can characterize translation-based block ciphers by the following:
\begin{definition}[\cite{CGC-cry-art-carantisalaImp}]\label{def:tb}
 A block cipher  $\mathcal{C} = \{ \gt_k \mid k  \in \mathcal{K} \}$ over
  ${\FF_2}$ is called {\em translation based (tb)} if:
       \begin{itemize}
    \item it is the composition of a finite number of rounds, such that any round $\gt_{k,h}$ can be written as $\gamma_h\gl_h\gs_{ k_h}$, where
    \begin{itemize}
    \item[-] $\gamma_h$ is a round-dependent bricklayer transformation (but it does not depend on $k$) and $0\gamma_h=0$,
     \item[-] $\lambda_h$ is a round-dependent linear map (but it does not depend on $k$),
    \item[-] ${k_h}$ is in $V$ and depends on both $k$ and the round (${k_h}$ is called a ``round key''),
     \end{itemize}
     
      \item for at least one round, called a ``proper'' round, we have (at the same time) that $\gl_h$ is proper and that the map $\Phi_h:\mathcal{K} \to V$ given by $k \mapsto k_h$ is
      surjective.
    \end{itemize}

\end{definition}

  The assumption $0\gamma_h = 0$ is not restrictive. Indeed, we can always include $0\gamma_h$ in the round key addition of the previous round (see \cite[Remark 3.3]{CGC-cry-art-carantisalaImp}).

Let $m\ge 1$, and let $f:(\FF_2)^m\rightarrow(\FF_2)^m$ be a vectorial Boolean function. We denote by $\hat{f}_u(x)=f(x+u)+f(x)$ the derivative of $f$ in the direction of $u\in(\FF_2)^m$. 
\begin{definition}
Let $f$ be a {vectorial Boolean function}, for any $a,b\in{(\FF_2)^m}$  we define
$$
\gd_f(a,b)=|\{x\in{(\FF_2)^m} \mid\, \hat{f}_a(x)=b\}|.
$$
Then, $f$ is said differentially $\gd$-uniform if 
$$
{\gd}=\max_{a,b \in{(\FF_2)^m},\atop
a\neq 0}\gd_f(a,b)\,.
$$
\end{definition}

Vectorial Boolean functions used as S-boxes in block ciphers must have low uniformity to prevent differential cryptanalysis (see \cite{CGC2-cry-art-biham1991differential}).
 By \cite[Fact 3]{CGC-cry-art-carantisalaImp}, a vectorial Boolean function differentially $\gd$-uniform satisfies
$$|{\rm Im}(\hat{f}_u)|\geq\frac{2^{m}}{\delta}.$$

\begin{definition}
Let $1\leq r<m$ and $f(0)=0$, we say that $f$ is {\it strongly $r$-anti-invariant} if, for any two subspaces $U$ and $W$ of $(\mathbb F_2)^m$ such that $f(U)=W$, then either $\dim(U)=\dim(W)<m-r$ or $U =W=(\mathbb F_2)^m$.
\end{definition}

\subsection{Partition-based trapdoors}
We recall that a permutation group $G$ acting on $V$ is called primitive if it has no nontrivial $G$-invariant partition of $V$.
That is, there no exists a partition $\cA$ of $V$ different from the trivial partitions $\{\{v\}\mid v \in V\}$, $\{V\}$, such that  $Ag\in \cA$ for all $A\in\cA$ and $g\in G$.
 On the other hand, if a nontrivial $G$-invariant partition exists, the group is called imprimitive.

As said before a property of $\Gamma_{\infty}$ considered undesirable is the imprimitivity. Paterson \cite{CGC-cry-art-paterson1} showed that if this group is imprimitive, then it is possible to embed a trapdoor in the cipher. 

Another trapdoor, based on the idea of the imprimitive action, is the Partition-based trapdoor, introduced in a recent work \cite{bannier2016partition}. In this work the authors give some conditions to construct a translation-based cipher which associates a partition of the plaintext space to another partition of the ciphertext  space.

We report some of the definitions and results presented in \cite{bannier2016partition}.

\begin{definition}
Let $\gr$ be a permutation of $V$ and $\cA, \cB$ be two partitions of $V$. Let $\cA \gr$ denote the set $\{A\gr \mid A \in \cA\}$. We say that $\gr$ maps $\cA$ to $\cB$ if $\cA\gr = \cB$. Moreover, let $G$ be a permutation group we say that $G$ maps $\cA$ to $\cB$ if for all $\gr\in G$, $\gr$ maps $\cA$ to $\cB$.
\end{definition}

\begin{remark}\label{rm:prim}
Note that a permutation group $G$ is imprimitive if there exists a non-trivial partition $\cA$ such that for all $\gr \in G$ $\cA\gr = \cA$.
\end{remark}

\begin{definition}
 A partition $\cA$ of $V$ is said linear if there exists $U$ a subspace of $V$ such that $$\cA=\{ U+v\mid v\in V\}.$$ We denote with $\cL(U)$ such a partition.
\end{definition}

The following result, introduced by Harpes in \cite{harpes1997partitioning}, characterizes the possible partitions $\cA$ and $\cB$ such that $T(V)$ maps $\cA$ to $\cB$.

\begin{proposition}
Let $\cA$ and $\cB$ be two partitions of $V=(\FF_2)^d$. Then the permutation group $T(V)$ maps $\cA$ to $\cB$ if and only if $\cA=\cB$ and $\cA$ is a linear partition.
\end{proposition}

Focusing  on the mixing-layer we have.

\begin{proposition}[\cite{bannier2016partition}]
Let $\gl$ be a linear permutation of $V$, i.e. $\gl\in \GL(V)$, and let $U$ be a subspace of $V$. Then $\cL(U)\gl=\cL(U\gl)$.
\end{proposition}

We report now the main theorem of \cite{bannier2016partition}. In \cite{bannier2016partition} the following result is reported for a SPN cipher with the same S-box and mixing-layer for each round, but can be extended to any translation based cipher with independent round keys.

\begin{theorem}\label{th:francesi}
Let $\cC$ be a translation based cipher on $V$. Suppose that there exist $\cA$ and $\cB$ non-trivial partitions such that for all $\ell$-tuples of round-keys $k=(k_1,...,k_\ell)$ the encryption function $\gt_k$ maps $\cA$ to $\cB$. Define $\cA_1=\cA$ and for $1\le i\le \ell$ $\cA_{i+1}=\cA_{i}\gt_{i}$, where $\gt_i=\gamma_i\gl_i$ is the $i$-th round function without the round key translation. Then
\begin{itemize}
\item $\cA_{\ell+1}=\cB$
\item for any $1\le i\le \ell+1$ $\cA_{i}$ is a linear partition.
\end{itemize}
\end{theorem}

\section{Avoiding the partition-based trapdoor}

In this section we will give sufficient conditions on the components of a tb cipher to guarantee that such a trapdoor cannot be implemented.

\begin{lemma}\label{lm:lm1}
Let $\gamma$ be a permutation on $V$ such that $0\gamma=0$ and suppose that $\gamma$ maps $\cL(U)$ to $\cL(W)$ then for all $u\in U$ $$\Im(\hat{\gamma}_u)\subseteq W.$$
\end{lemma}
\begin{proof}
The fact that $\gamma$ maps $\cL(U)$ to $\cL(W)$ and $0\gamma=0$ imply that $U\gamma=W$. Moreover for all $v\in V$ we have that $v\gamma \in (U+v)\gamma$. Then $(U+v)\gamma=W+v\gamma$. This implies that $$(u+v)\gamma+v\gamma\in W$$ for all $v\in V$ and $u\in U$.
\end{proof}

From Lemma \ref{lm:lm1} we have the following.

\begin{proposition}\label{prop:blocchi}
Let $\gamma$ be a brick-layer transformation, i.e. $\gamma=(\gamma_1,...,\gamma_b)$ with $\gamma_i\in\Sym(V_i)$ for all $i$. Suppose that for all $i$ $\gamma_i$ is
\begin{enumerate}
\item  differentially $2^r$-uniform, with $r<m$, 
\item strongly $(r-1)$-anti-invariant 
\end{enumerate}
Let $\cL(U)$ and $\cL(W)$ be non-trivial linear partitions of $V$, then $\gamma$ maps $\cL(U)$ to $\cL(W)$ if and only if $U$ and $W$ are walls, in particular $U=W$.
\end{proposition}
\begin{proof}
Suppose that $\gamma$ maps $\cL(U)$ to $\cL(W)$ and that $U$ is not a wall. Then let $I=\{i\mid \pi_i(U)\ne 0\}$. There exists $i\in I$ such that $U\cap V_i\ne V_i$ because $U$ is not a wall. Moreover $U\cap V_i\ne \{0\}$. Indeed, let $u\in U$, with $u_i\ne0$. For all $v\in V$ such that $v_j=0$ for all $j\ne i$ we have that $$(u+v_i)\gamma+v_i\gamma\in W$$
from Lemma \ref{lm:lm1}. Moreover $u\gamma \in W$. It follows that $u\gamma+(u+v_i)\gamma+v_i\gamma\in W$. 
The vector $u\gamma+(u+v_i)\gamma+v_i\gamma$ has all nonzero components but for the one in $V_i$, which is $u_i\gamma_i+(u_i+v_i)\gamma_i+v_i\gamma_i \in W \cap V_i$. As $\gamma$ is a brick layer transformation, if $U\cap V_i= \{0\}$, then $W\cap V_i= \{0\}$. So, if the vector $u_i\gamma_i+(u_i+v_i)\gamma_i+v_i\gamma_i$  is zero for all $v_i \in V_i$, then $\Im(\hat{\gamma}_i)=\{ u_i\gamma_i \}$, of size 1. This contradicts the first condition on the $\gamma_i$'s.

Thus  $U\cap V_i$ is a proper subspace of $V_i$ and $(U\cap V_i)\gamma=W\cap V_i$ is also a proper subspace of $V_i$. Since $\Im(\hat{\gamma}_u)\subseteq W$ for all $u\in U$, we have $\Im(\hat{\gamma_i}_u)\subseteq W\cap V_i$ for all $u\in U\cap V_i$. Thus, $|\Im(\hat{\gamma_i}_u)|$ for any non-zero $u$ is greater or equal to $2^{m-r}$ and, being $\gamma_i$ a permutation, the set $\Im(\hat{\gamma_i}_u)$ does not contain the zero vector, which implies $\dim(W\cap V_i)\ge m-r+1$. This contradicts the second condition on the $\gamma_i$'s. Then, $U$ is a wall and, since $\gamma$ is a parallel map, we have $U=W$.

Vice versa, consider any wall $U$, as $\gamma$ is a parallel map, it is easy to check that $\gamma$ maps $\{U+v\mid v\in V\}$ in itself.
\end{proof}

The following definition was introduced in \cite{calderinisalerno}. 

\begin{definition}
Let $\gl\in \GL(V)$. $\gl$ is said to be a strongly proper mixing-layer if it does not map a proper wall in an other wall.
\end{definition}

We define a strongly-proper round of a tb cipher as a proper round, where the mixing-layer is also strongly proper.

\begin{theorem}\label{th:main}
Let a round $h<\ell$ be a strongly proper round and suppose that the brick-layer transformations of round $h$ and round $h+1$, $\gamma_h$ and $\gamma_{h+1}$, satisfy Condition 1) and 2) of Proposition \ref{prop:blocchi}. Then, there do not exist $\cA$ and $\cB$ non-trivial partitions such that for all $\ell$-tuples of round keys the encryption functions map $\cA$ to $\cB$.
\end{theorem}
\begin{proof}
Suppose that the partition-based trapdoor is applicable for all $\ell$-tuples of round keys. From Theorem \ref{th:francesi} we have that there exist two linear partitions $\cL(U)$ and $\cL(W)$ such that $\gamma_h\gl_h$ maps $\cL(U)$ to $\cL(W)$. 
Thus $\gamma_h$ maps $\cL(U)$ to $\cL(W(\gl_h)^{-1})$. From Proposition \ref{prop:blocchi} we have that $U$ is a wall and the same for $W(\gl_h)^{-1}$. Now being $\gl_h$ strongly-proper we have that $W$ is not a wall. Then from Theorem \ref{th:francesi} we have that $\gamma_{h+1}$ maps $\cL(W)$ to another linear partition, but it is not possible since $W$ is not a wall and this contradicts Proposition \ref{prop:blocchi}.
\end{proof}

The result of Theorem \ref{th:main} can be generalized, using a weaker condition on the mixing-layers composing the round functions. Consider, for example, the mixing-layer of AES. This is not strongly-proper, as there exists a wall which is sent in another wall. Indeed, the state of AES can be represented as a $4 \times 4$ matrix of bytes (Table \ref{tb:1}). That is, $V=V_1\oplus...\oplus V_{16}$, where $V_i$ is isomorphic to $(\FF_2)^8$.

\begin{table}[h!]
\centering
\begin{tabular}{|c|c|c|c|}
\hline
$V_1$&$V_2$&$V_3$&$V_4$\\
\hline
$V_5$&$V_6$&$V_7$&$V_8$\\
\hline
$V_9$&$V_{10}$&$V_{11}$&$V_{12}$\\
\hline
$V_{13}$&$V_{14}$&$V_{15}$&$V_{16}$\\
\hline
\end{tabular}\caption{AES state}\label{tb:1}
\end{table}

The mixing-layer of AES is composed by two linear functions ShiftRow (SR) and MixColumn (MC).
The ShiftRows transformation acts in a way such that a wall is sent in an other wall, and in particular $V_1\oplus V_{6}\oplus V_{11}\oplus V_{16}$ is sent in $V_1\oplus V_{5}\oplus V_{9}\oplus V_{13}$.
\begin{table}[h]
\centering
\begin{tabular}{|c|c|c|c|}
\hline
\cellcolor{orange}$V_1$&$V_2$&$V_3$&$V_4$\\
\hline
$V_5$&\cellcolor{orange}$V_6$&$V_7$&$V_8$\\
\hline
$V_9$&$V_{10}$&$\cellcolor{orange}V_{11}$&$V_{12}$\\
\hline
$V_{13}$&$V_{14}$&$V_{15}$&\cellcolor{orange}$V_{16}$\\
\hline
\end{tabular}${SR\atop\mapsto}$
\begin{tabular}{|c|c|c|c|}
\hline
\cellcolor{orange}$V_1$&$V_2$&$V_3$&$V_4$\\
\hline
\cellcolor{orange}$V_5$&$V_6$&$V_7$&$V_8$\\
\hline
\cellcolor{orange}$V_9$&$V_{10}$&$V_{11}$&$V_{12}$\\
\hline
\cellcolor{orange}$V_{13}$&$V_{14}$&$V_{15}$&$V_{16}$\\
\hline
\end{tabular}
${MC\atop\mapsto}$
\begin{tabular}{|c|c|c|c|}
\hline
\cellcolor{orange}$V_1$&$V_2$&$V_3$&$V_4$\\
\hline
\cellcolor{orange}$V_5$&$V_6$&$V_7$&$V_8$\\
\hline
\cellcolor{orange}$V_9$&$V_{10}$&$V_{11}$&$V_{12}$\\
\hline
\cellcolor{orange}$V_{13}$&$V_{14}$&$V_{15}$&$V_{16}$\\
\hline
\end{tabular}\caption{AES wall}\label{tb:2}
\end{table}
Now, as MixColumn combines the blocks of a same column of the state, we have that $V_1\oplus V_{5}\oplus V_{9}\oplus V_{13}$ (which is the first column of the state) is sent in itself (Table~\ref{tb:2}).
However, the previous attack cannot be applied to AES as the mixing layer satisfies the following property.

\begin{definition}\label{def:stro}
Let $\gl_1,...,\gl_\ell$ be the mixing-layers used in a tb cipher with $\ell$ rounds. We say that the ordered family $\Lambda=(\gl_1,...,\gl_\ell)$ is strongly proper if for all possible non-trivial wall $W$ there exists $j< \ell$ such that $W\gl_1\cdot...\cdot \gl_j$ is not a wall.
\end{definition}

\begin{corollary}\label{cor:1}
Let $\cC$ be a tb cipher with $\Lambda=(\gl_1,...,\gl_\ell)$ strongly proper and for each round $h$ the parallel map $\gamma_h$ satisfies the conditions of Proposition \ref{prop:blocchi}. Then the partition-based trapdoor is not applicable.
\end{corollary}
\begin{proof}
Suppose that the partition-based trapdoor is applicable. From Theorem \ref{th:francesi} we have that each round functions (without the translation with respect to the round key) $\gt_h=\gamma_h\gl_h$ maps a linear partition $\cL(U_h)$ into $\cL(U_{h+1})$. From Lemma \ref{lm:lm1} the space $U_h$ is a wall for all $1\le h\le \ell$. Then, being $\gamma_h$ a parallel map and $0\gamma_h=0$ we have $$U_{h+1}=U_{h}\gl_h=U_1\gl_1\cdot...\cdot \gl_h$$ for all $h$. 

Since $\Lambda$ is strongly proper, there exists $j<\ell$ such that $U_1\gl_1\cdot...\cdot \gl_j$ is not a wall, which contradicts the fact that $U_h$ is a wall for all $1\le h\le \ell$.
\end{proof}

It may seem that the previous corollary requires strong conditions on the mixing layers and on the S-boxes. We show that the requirement on the mixing layers is necessary.

\begin{proposition}
Let $\cC$ be a tb cipher, with $\Lambda=(\gl_1,...,\gl_\ell)$ not strongly proper. Then the partition trapdoor is applicable.
\end{proposition}
\begin{proof}
Because $\Lambda=(\gl_1,...,\gl_\ell)$ is not  strongly propers, then there exists $W$ a proper wall such that $W\gl_1\cdot...\cdot \gl_j$ is a wall for all $1\le j\le \ell$. Now being the $\gamma_h$ a parallel map we have that for all wall $W'$ and $v\in V$ $$(W'+v)\gamma_h=W'+v\gamma_h,$$ which concludes the proof.
\end{proof}
\section{Security proof for a cipher with independent round-keys}\label{sec:ind}
In this section we will discuss the fact that studying the group $\Gamma_\infty$ to understand the security of a block cipher $\cC$ could not be enough. \\

We report an example of block cipher $\cC$ whose components satisfy the cryptographic properties given in \cite{calderinisalerno} sufficient to thwart the trapdoor introduced by Paterson \cite{CGC-cry-art-paterson1}, but that is weak with respect to the partition-based trapdoor. 

Let $V=(\FF_2)^{mb}$, $V_i=(\FF_2)^m$, with $m,b\ge 1$ and $\gamma_i\in \Sym(V_i)$ be the inverse permutation for all $0\le i\le b-1$, i,e. $\gamma_i:x\mapsto x^{2^m-2}$ (using the representation as univariate polynomial). Consider the following mixing-layer 
$$
\gl=\left[\begin{array}{ccccc}
0 & I_{m\times m} & 0 &...& 0\\
0& 0 & I_{m\times m} &...& 0\\
\vdots & & & &\vdots\\
0 & &...& &  I_{m\times m}\\
 I_{m\times m} & &...& & 0\end{array}
\right]
$$
where $ I_{m\times m} $ is the identity matrix of size $m\times m$. It is easy to check that $\gl$ is a proper mixing-layer but not strongly proper. 

Consider now the linear partitions $\cL(V_1),...,\cL(V_b)$, and suppose to have $\ell$ rounds where we use $\gl$ and $\gamma$ in each one. From \cite[Theorem 3.1]{calderinisalerno} $\Gamma_\infty(\cC)$ is primitive, but each encryption function maps $\cL(V_i)$ to $\cL(V_{\gs^\ell (i)})$, where $\gs$ is the permutation of $\{1,...,b\}$ such that $\gs:i\mapsto i+1$ for all $i<b$ and $\gs:b\mapsto 1$.\\
Obviously, the mixing-layer $\gl$ is not interesting from a cryptographic point of view. We use this only as an example to show that if $\Gamma_\infty(\cC)$ is primitive this does not guarantee security on $\cC$. 
Moreover, if we use a number $\ell$ of rounds such that $\cL(V_i)=\cL(V_{\gs^\ell (i)})$ for some $i$, then $\Gamma_\infty(\cC)$ is primitive and the group $\Gamma(\cC)$ is imprimitive.

As we pointed out in Section 2, it would be interesting to study the group $\Gamma(\cC)$. However, this group depends, strongly, on the key-schedule used to create the round-keys. For this reason, usually, we study the properties of the group $\Gamma_\infty(\cC)$ (see for instance \cite{GOST,even,Wernsdorf2,Wernsdorfserpent,Wernsdorfaes}). But, as we showed above, we could have that even if the group $\Gamma_\infty(\cC)$ is considered secure, the group $\Gamma(\cC)$ may be not secure with respect to the trapdoors considered here. Then, we think that it is better to consider, and study, the group of a cipher $\cC$ obtained using independent round-keys, that is,
$$
\Gamma_{ind}(\cC)=\langle \gt_{K}\mid K\in V^\ell\rangle,
$$
where, letting $K=(k_1,\dots,k_\ell)$, $\gt_K$ is the encryption function obtained using $k_h\in V$ as round-key at round $h$.
Clearly, $\Gamma_{ind}(\cC)$ is such that $$\Gamma(\cC)\subseteq\Gamma_{ind}(\cC)\subseteq \Gamma_\infty(\cC).$$

We can summarize the results obtained in this work for $\Gamma_{ind}(\cC)$ in the following:
\begin{theorem}\label{th:main1}
Let $\cC$ be a tb cipher. Suppose that one of the following properties is satisfied:
\begin{enumerate}
\item there exists a round $h$ which is a strongly proper round and the brick-layer transformations $\gamma_h$, $\gamma_{h+1}$ satisfy Condition 1) and 2) of Proposition \ref{prop:blocchi},
\item[] \begin{center} or \end{center}
\item the family $\Lambda=(\gl_1,...,\gl_\ell)$ is strongly proper and for each round $h$ the parallel map $\gamma_h$ satisfies Condition 1) and 2) of Proposition \ref{prop:blocchi}.
\end{enumerate}
Then the partition-based trapdoor is not applicable. Moreover, $\Gamma_{ind}(\cC)$ is primitive.
\end{theorem}
\begin{proof}
From Theorem \ref{th:main} (or Corollary \ref{cor:1}) we have that there do not exist two partitions $\cA$ and $\cB$ such that $\cA\gt_k=\cB$ for all $\gt_k\in\Gamma_{ind}(\cC)$. Moreover, the group $\Gamma_{ind}(\cC)$ is imprimitive if and only if  there exists a partition such that $\cA\gt_k=\cA$ for all $\gt_k\in\Gamma_{ind}(\cC)$ (see Remark \ref{rm:prim}), which is a particular case where the partition-based trapdoor is applicable.
\end{proof}

\section{Conclusions and final remarks}

An interesting open problem proposed by Paterson \cite{CGC-cry-art-paterson1} was to investigate if it is possible to construct a block cipher such that the group $\Gamma_\infty$ is primitive, but the resulting cipher is weak with respect to the imprimitive trapdoor. As we pointed out in Section \ref{sec:ind}, it may happen that the cipher is vulnerable even if the group generated by the round functions results to be primitive. For this reason, here, we studied the group $\Gamma_{ind}$ generated by the cipher with independent round keys.  We think that studying algebraic properties for this group could be more appropriate to investigate security properties of a cipher. In particular, we gave sufficient conditions on the components of a cipher $\cC$, so that the partition trapdoor given in \cite{bannier2016partition} cannot be applied to the encryption functions generating $\Gamma_{ind}(\cC)$. As a consequence, we obtained also the primitivity for the group $\Gamma_{ind}(\cC)$. 

Note that this type of trapdoor can be easily avoided. Indeed, as noted in \cite{calderinisalerno}, for an invertible vectorial Boolean function to be strongly $1$-anti-invariant is equivalent to have no linear components (i.e. the nonlinearity is greater than $0$). Such a property is usually requested by the S-boxes of a cipher to avoid linear cryptanalysis \cite{linear}. Moreover, to achieve a good \emph{diffusion} we need that the mixing layer satisfies the condition given in Definition \ref{def:stro}. Therefore, from Theorem \ref{th:main1}, sufficient conditions to thwart the partition trapdoor are:

\begin{itemize}

\item S-boxes with differential $4$-uniformity and nonlinearity different from $0$.

\item strongly proper mixing layers (or that satisfy the condition in Definition~\ref{def:stro}).

\end{itemize}
{
For well-know ciphers like AES, PRESENT and SERPENT we have that these two characteristics are satisfied.
}\\

{In the work \cite{bannier2017partition}, the authors 
give a method to construct the S-boxes of a block-cipher in order to let be possible the partition trapdoor (see Section 3 in  \cite{bannier2017partition}). Using their method they construct a cipher, with S-boxes defined over $\mathbb{F}_{2}^{10}$, which is weak with respect to their trapdoor. It is possible to check that such a cipher cannot be attacked using the classical linear and differential cryptanalysis. However, 
as the same author stated, the structure of their linear and differential tables is likely to betray the existence of a backdoor and can be used to find it.
For this reason they create an other attack perturbing the S-boxes of the cipher, in order to strengthen it. These new S-boxes ``behave'' similarly to their secret counterparts, that means the output of a perturbed S-box is with high probability equals to the output of the corresponding non-perturbed S-box.}

{The mixing layer used in this block cipher is similar to the AES mixing layer, in particular we can check that the condition in Definition \ref{def:stro} is satisfied. The perturbed S-boxes are at most differentially $40$-uniform (and thus $2^5$-uniform). So, if we would analize if the properties of Theorem \ref{th:main1} are satisfied we need to check if it is strongly $4$-anti-invariant. This computation could be possible, but quite long. So we could verify the resistance of the block cipher to the partition trapdoor in an other way.}

{In Theorem \ref{th:main1} we use the differential uniformity of the S-boxes to say that $\Im(\hat{\gamma}_u)$ is greaten than $2^{n}/2^r$. So, we could generalize the first condition in Proposition \ref{prop:blocchi} as follow.
\begin{itemize}
\item[(1')]  $\Im(\hat{\gamma}_u)>2^{n-r}$ for all nonzero $u\in \mathbb{F}_{2}^{n}$.
\end{itemize}
Note that this condition is pretty similar to the {\em weak differential uniformity} introduced in \cite{caranti2} and also studied in \cite{weak}.}

{Using this fact, we have, for the case of the perturbed S-boxes of the cipher, that $\Im(\hat{\gamma}_u)>2^{6}$. So, we can check if the the S-boxes are strongly $3$-anti-invariant. This check is much easier, indeed we need to check if vector subspaces of dimension at least $7$ are mapped onto an other vector subspace. Using the software MAGMA it is possible, now, to check that all the S-boxes of the cipher are strongly $3$-anti-invariant. Which implies that the cipher is secure with respect to the trapdoor. Then, using only the analysis reported in this work, we cannot detect the backdoor introduced in \cite{bannier2017partition}. However the high differential uniformity of the S-boxes could lead to some suspect, suggesting the existence of a backdoor.}

{So it is a very interesting open problem to understand which cryptographic properties can be useful to avoid also the attack with perturbed S-boxes.}\\

Another requested property for a cipher is that the group generated by the encryption functions is not ``small'' (see for instance \cite{Kaliski}). Usually, this property is investigated for the group of the round functions (\cite{GOST,even,Wernsdorf2}).
With an ad hoc proof, in \cite{Wernsdorfaes} and \cite{Wernsdorfserpent} Wernsdorf proved respectively that $\Gamma_{\infty}(AES) = \Alt(V )$ and $\Gamma_{\infty}(SERPENT) = \Alt(V )$, where $\Alt(V )$ is the alternating group. 
In \cite{calderinisalerno,caranti2} are given conditions for tb ciphers so that $\Gamma_\infty(\cC)$ is the symmetric (alternating) group. 

In the case of ciphers like AES and PRESENT, where the same S-box and mixing layer are used in each round, we have that $\Gamma_{ind}(\cC)$ is normal in $\Gamma_{\infty}(\cC)$ (see for instance \cite[Lemma 3.4]{GOST}), so if $\Gamma_{\infty}(\cC)$ is the symmetric (alternating) group we have $\Gamma_{ind}(\cC)=\Gamma_{\infty}(\cC)$. However, this is no more the case of ciphers like, e.g., SERPENT, where the used S-box depends on the round.
Then, it may happen that  the group of the round functions generates the symmetric (alternating) group, while the group generated by the encryption functions is not. 

{Moreover, the only condition that $\Gamma_\infty(\cC)$ is the symmetric (alternating) group does not guarantee that the cipher is secure with respect the trapdoor studied in this work. Indeed, if we consider the example of the cipher given in Section 4, adding at the last round a strongly-proper mixing layer we will obtain that $\Gamma_{\infty}(\cC)$ is the alternating group. This is implied by Theorem 4.7 in \cite{calderinisalerno}. However, as in Section 4, such a cipher can be broken using the partition based trapdoor. Another example of a weak cipher having $\Gamma_{\infty}(\cC)=\Sym(V)$ is given in \cite{mur}.}

Another interesting future research could be investigating assumptions on the components of a cipher $\cC$ which can guarantee that the group $\Gamma_{ind}(\cC)$ is the symmetric (alternating) group.

%
%




\medskip
Received xxxx 20xx; revised xxxx 20xx.
\medskip

\end{document}